\documentclass[letterpaper, 10 pt, conference]{ieeeconf}  %

\IEEEoverridecommandlockouts                              %

\usepackage{cite}

\usepackage{subcaption}
\usepackage{tikz}

\usepackage{amsmath}
\usepackage{array}
\usepackage{url}

\DeclareMathAlphabet{\mathcal}{OMS}{cmsy}{m}{n} %
\usepackage{algorithm}%
\usepackage{algpseudocode}%
\usepackage{siunitx}
\usepackage{import} %
\usepackage{xcolor}
\usepackage{amsmath}
\usepackage{amssymb}  %
\usepackage{multirow}
\usepackage{multicol}
\usepackage{booktabs}

\newtheorem{assumption}{Assumption}
\newtheorem{lemma}{Lemma}
\newtheorem{proposition}{Proposition}
\newtheorem{theorem}{Theorem}
\newtheorem{corollary}{Corollary}

\newtheorem{definition}{Definition}

\newtheorem{result}{Result}
\usepackage{bm}
\usepackage{hyperref}

\def\eps{{\epsilon}}

\def\vc{{c}}

\def\vu{{u}}

\def\vw{{w}}
\def\vx{{x}}

\def\mA{{A}}
\def\mB{{B}}

\def\mK{{K}}

\DeclareMathAlphabet{\mathsfit}{\encodingdefault}{\sfdefault}{m}{sl}
\SetMathAlphabet{\mathsfit}{bold}{\encodingdefault}{\sfdefault}{bx}{n}

\def\gA{{\mathcal{A}}}

\def\gP{{\mathcal{P}}}

\def\gU{{\mathcal{U}}}

\def\sD{{\mathbb{D}}}

\def\sP{{\mathbb{P}}}

\def\sR{{\mathbb{R}}}

\def\sX{{\mathbb{X}}}

\newcommand{\E}{\mathbb{E}}

\newif\ifcomments\commentsfalse
\newif\ifshowold\showoldfalse
\newif\ifshowreviewcomments\showreviewcommentsfalse
\newif\ifshowproof\showprooftrue

\newcommand{\OLD}[1]{\ifshowold {\textcolor{brown}{OLD: #1}} \else %
    \fi}

\newcommand{\PROOFS}[1]{\ifshowproof { #1} \else %
    \fi}

\newcommand{\Footnote}[1]{\ifshowproof 
\else { #1}
    \fi}

\newcommand\copyrighttext{%
		\footnotesize \copyright 2025 IEEE. Personal use of this material is permitted. Permission from IEEE must be
obtained for all other uses, in any current or future media, including
reprinting/republishing this material for advertising or promotional purposes, creating new
collective works, for resale or redistribution to servers or lists, or reuse of any copyrighted
component of this work in other works.
}

\newcommand\copyrightnotice{%
	\begin{tikzpicture}[remember picture,overlay]
		\node[anchor=south,yshift=10pt] at (current page.south) {\fbox{\parbox{\dimexpr\textwidth-\fboxsep-\fboxrule\relax}{\copyrighttext}}};
	\end{tikzpicture}%
}

\newcommand{\blista}{\renewcommand{\labelenumi}{(\roman{enumi})} 
\begin{enumerate}}
\newcommand{\elista}{\end{enumerate} \renewcommand{\labelenumi}{\arabic{enumi}.}}

\title{\LARGE \bf
On the risk levels of distributionally robust chance constrained problems
}

\author{Moritz Heinlein$^{1}$, Teodoro Alamo$^{2}$,
        and Sergio Lucia$^{1}$%
\thanks{*The research leading to these results has received funding from the DFG (German Research Foundation) under grant agreement number 423857295.
T. Alamo acknowledges support from grants PID2022-142946NA-I00 and PID2022-141159OB-I00, funded by MICIU/AEI/ 10.13039/501100011033 and by ERDF/EU. The authors used ChatGPT (\url{https://chat.openai.com/}) for alternative formulations.(\emph{Corr. author: M. Heinlein})}%
\thanks{$^{1}$ M. Heinlein and S. Lucia are with the Chair of Process Automation Systems at Technische Universität Dortmund, Emil-Figge-Str. 70, 44227 Dortmund, Germany (email: $\{ \text{moritz.heinlein}\}$, $\{\text{sergio.lucia}\}$ @tu-dortmund.de).}%
\thanks{$^2$ T. Alamo is with the Department of Systems Engineering and Automation, Universidad de Sevilla, Sevilla, Spain (e-mail: talamo@us.es).} %
}

\begin{document}

\maketitle
\PROOFS{\copyrightnotice}
\thispagestyle{empty}
\pagestyle{empty}

\begin{abstract}

In this paper, we discuss the utilization of perturbed risk levels (PRLs) for the solution of chance-constrained problems via sampling-based approaches. PRLs allow the consideration of distributional ambiguity by rescaling the risk level of the nominal chance constraint.
Explicit expressions of the PRL exist for some discrepancy-based ambiguity sets.

We propose a discrepancy functional not included in previous comparisons of different PRLs based on the likelihood ratio, which we term ,,relative variation distance" (RVD). If the ambiguity set can be described by the RVD, the rescaling of the risk level with the PRL is in contrast to other discrepancy functionals possible even for very low risk levels. 
We derive distributionally robust one- and two-level guarantees for the solution of chance-constrained problems with randomized methods.
We demonstrate the viability of the derived guarantees for a randomized MPC under distributional ambiguity.

\end{abstract}

\section{Introduction}
Chance constrained formulations~\cite{millerChanceConstrainedProgramming1965, tempoRandomizedAlgorithmsAnalysis2013} allow a certain probability of constraint violation, known as the \emph{risk level}. Their formulation requires an accurate knowledge of the probability distribution of the underlying uncertainty, and even then the problem can be difficult to solve.

Randomized methods such as the scenario approach~\cite{calafioreScenarioApproachRobust2006, campiIntroductionScenarioApproach2018} or statistical learning~\cite{alamoRandomizedStrategiesProbabilistic2009} approximate chance constrained problems by sampling from the uncertainty distribution. Especially in early design stages without measurements, this distribution may be poorly known.

Distributionally robust optimization tackles this challenge by considering all uncertainties within an ambiguity set~\cite{rahimianFrameworksResultsDistributionally2022}. Discrepancy-based ambiguity sets are described as the set of all probability distributions within a certain distance around a nominal distribution.
There exists a range of discrepancy functionals used to formulate distributionally robust optimization problems~\cite{rahimianFrameworksResultsDistributionally2022}. 
Other descriptions of ambiguity sets include moment-based~\cite{mcallisterDistributionallyRobustModel2024,liDistributionallyRobustOptimization2022}, shape-preserving~\cite{nemirovskiConvexApproximationsChance2007}, and kernel-based models (see~\cite{rahimianFrameworksResultsDistributionally2022} for a review).

For some discrepancy functionals, mostly for continuous distributions, only the risk level of the chance constraint needs to be adapted to also guarantee distributionally robust results. The adapted risk level is called the \emph{perturbed risk level} (PRL)~\cite{jiangDatadrivenChanceConstrained2016,tsengRandomConvexApproximations2016a, hongAmbiguousProbabilisticPrograms2013}. Therefore, the distributionally robust program with a PRL has the same complexity as the original chance constrained problem and does not require the calculation of Lipschitz constants, in contrast to Wasserstein-distance based backoffs~\cite{kuhnWassersteinDistributionallyRobust2019a}.
In the work of \cite{erdoganAmbiguousChanceConstrained2006}, the ambiguity in the solution of randomized problems~\cite{calafioreScenarioApproachRobust2006} was considered with respect to the Prohorov distance, which is difficult to evaluate even for common distributions~\cite{gibbsChoosingBoundingProbability2002}.
In~\cite{jiangDatadrivenChanceConstrained2016} and~\cite{tsengRandomConvexApproximations2016a} PRLs are presented to guarantee chance constraint satisfaction under ambiguity measured by different $\phi$-divergences. 
For the Wasserstein distance, a data-driven reformulation of affine chance constraints leads to mixed-integer conic programs~\cite{chenDataDrivenChanceConstrained2024}, but no direct PRL.
In~\cite{nemirovskiConvexApproximationsChance2007, hongAmbiguousProbabilisticPrograms2013}, a PRL was calculated based on the upper bound of the likelihood ratio between the ambiguous distributions and a nominal distribution. However, the potential of this metric has not been further explored in later works.~\cite{tsengRandomConvexApproximations2016a,jiangDatadrivenChanceConstrained2016}.%
In this paper, we term the corresponding discrepancy functional \emph{relative variation distance} (RVD).

This paper contains two main contributions. First, we propose the RVD as an especially well-suited distance for chance constraints with low risk levels, as it is desired in most control applications. For other discrepancy functionals, the PRL declines rapidly for small risk levels. To underline the applicability of the RVD, we derive a closed expression for the RVD between Gaussians.
Second, we derive distributionally robust one-level and two-level guarantees on the violation probability for the data-based solution of chance constraints. We show their validity for the application of the scenario approach in randomized MPC.
 Section~\ref{sec:PRL} introduces the distributionally robust setting, as well as PRLs and the RVD. Section~\ref{sec:Comp_of_PRL} compares the RVD to PRLs based on other discrepancy functionals for small risk levels.
Section~\ref{sec:application} includes the derivation of distributionally robust guarantees of the violation probability for the solution of chance constraints via randomized methods.
In Section~\ref{sec:Case_Study} the derived guarantees are tested for a randomized MPC.

\emph{Notation:} We denote that the random variable 
 $\delta$ follows the distribution $\mathcal{P}$ by writing 
 $\delta \sim \mathcal{P}$. The probability of an event $E$ according to $\mathcal{P}$ is written as $\sP_{\mathcal{P}}(E)$.
The probability density function of the distribution $\mathcal{P}$ is denoted as $f_{\mathcal{P}}:\Delta_{\mathcal{P}} \rightarrow \sR^+ $, where $\Delta_{\mathcal{P}}$ is the support of $\mathcal{P}$, which is the smallest set satisfying $\sP_{\mathcal{P}}(\delta\in\Delta_{\mathcal{P}})=1$. With $\Sigma\prec\Omega$, the positive definiteness of the matrix $\Omega-\Sigma$ is implied. The indicator function $I(x):\sR\rightarrow \{0,1\}$ is 1 if $x\geq0$, else 0.

\section{Perturbed risk level}\label{sec:PRL}
We consider chance constrained optimization problems
\begin{subequations} \label{eq:CCP}
\begin{align}
    \min_{\vx\in\sX} \,& \ J(\vx)\\
    \text{s.t.}\quad & \sP_{ \mathcal{P}} \left(g_{\delta}(\vx)\leq 0\right) \geq 1-\epsilon, \label{eq:CCP:subeq:CC}
\end{align}
\end{subequations}
where $\delta\in\Delta_{\mathcal{P}}\subseteq\sR^{n_d}$ is a randomly distributed variable according to $\mathcal{P}$ and $g_{\delta}(\vx)$ is a constraint depending on the uncertainty $\delta$. %
The risk level, defined as the upper bound on the violation probability of $g_{\delta}(\vx)\leq 0 $, is specified as $\epsilon$.

To consider ambiguity in the distribution of the random variable $\delta$, the distributionally robust chance constrained optimization problem is formulated as:
\begin{subequations} \label{eq:DCCP}
\begin{align}
    \min_{\vx\in\sX} \,& \ J(\vx)\\
    \text{s.t.}\quad & \inf_{\mathcal{P}\in\mathcal{A}} \{\sP_{\mathcal{P}} \left(g_{\delta}(\vx)\leq 0\right)\} \geq 1-\epsilon, \label{eq:DCCP:subeq:CC}
\end{align}
\end{subequations}
where the set $\mathcal{A}$ describes the set of possible probability distributions, called the ambiguity set.
Discrepancy-based ambiguity sets include all probability distributions which are close to a nominal distribution $\hat{\mathcal{P}}$ with respect to some discrepancy functional $\rho$
\begin{equation} \label{eq:size_of_amb_set}
    \mathcal{A} = \{ \mathcal{P}|\rho(\mathcal{P}, \hat{\mathcal{P}}) \leq M\},
\end{equation}
where $M\in\sR$ bounds the size of the ambiguity set~\cite{rahimianFrameworksResultsDistributionally2022}.
Some discrepancy functionals can only yield finite distances for distributions $\mathcal{P}$, for which every possible sample is also a possible sample from the nominal distribution $\hat{\mathcal{P}}$. This is covered in the following assumption, which requires absolute continuity with respect to $\hat{\mathcal{P}}$ (for more details see~\cite{follandSignedMeasuresDifferentiation1999}).
\begin{assumption} \label{ass:greater_zero}
    The support of the nominal distribution $\Delta_{\hat{\mathcal{P}}}$ covers the support $\Delta_{\mathcal{P}}$ of all possible distributions in the ambiguity set $\mathcal{A}$:
    $$ \Delta_{\mathcal{P}} \subseteq \Delta_{\hat{\mathcal{P}}} \subseteq{\Delta}, \forall  \mathcal{P} \in \mathcal{A}.$$
\end{assumption}
For readability, we denote the governing support as $\Delta$ and only refer to the specific support if necessary.
There are many distributions, like the normal distribution, for which Assumption~\ref{ass:greater_zero} is satisfied naturally.

In the context of distributionally robust optimization, it is important to be able to bound the probability of any event $E\subseteq \Delta$, under the assumption that the underlying probability distribution $\mathcal{P}$ belongs to an ambiguity set $\mathcal{A}$. That is, one is interested in bounding $\max_{\mathcal{P}\in \mathcal{A}}\sP_{\mathcal{P}}\{E\}$. In~\cite{jiangDatadrivenChanceConstrained2016}, a rescaling of the nominal risk level to account for distributional ambiguity was termed perturbed risk level (PRL). This concept is stated formally in the following definition.
 \begin{definition}[Perturbed Risk Level] \label{def:PRL} 
Given  the ambiguity set $\mathcal{A}$ on the governing support $\Delta$ and the nominal distribution $\hat{\mathcal{P}}\in \mathcal{A}$,  the PRL 
     $\hat{\epsilon}_{\mathcal{A}}: [0,1]\to[0,1]$ is defined as
\begin{multline} \label{eq:PRL:Def}
\hat{\epsilon}_{\mathcal{A}} (\epsilon) := \max \{ \alpha : \sP_{\hat{\mathcal{P}}}\{E\} \leq \alpha\\
        \Rightarrow \sup_{\mathcal{P}\in \mathcal{A}}\sP_{\mathcal{P}}\{E\} \leq \epsilon, \forall E \subseteq \Delta  \}.
        \end{multline}
\end{definition}
We will denote the PRL for ambiguity sets of the form~\eqref{eq:size_of_amb_set} as $\hat{\epsilon}_{M_{\rho}}(\epsilon)$, where $M$ denotes the radius of the ambiguity set with respect to the discrepancy functional $\rho$. If the context is clear, we omit the dependence of $\hat{\epsilon}_{\mathcal{A}}$ on $\epsilon$.
For notational convenience, we denote the probability of violation of the constraint $g_\delta(x)\leq 0$ for given probability distribution $\mathcal{P}$ and $x$ as $V_{\mathcal{P}} (x) = \sP_{\mathcal{P}}\{\delta\in\Delta: g_{\delta}(x)> 0\}$.

The following lemma, which is a direct consequence of the definition of the PRL, provides a way to bound the probability of violation of a given constraint in an ambiguity set for~\eqref{eq:DCCP}~\cite{tsengRandomConvexApproximations2016a,jiangDatadrivenChanceConstrained2016}.

\begin{lemma}\label{lemma:PRL:VN}
Suppose $\hat{\epsilon}_{\mathcal{A}}: [0,1]\to[0,1]$  is a  PRL for the ambiguity set $\mathcal{A}$ and the nominal distribution $\hat{\mathcal{P}}\in \mathcal{A}$. Then,
$$  V_{\hat{\mathcal{P}}} (x) \leq \hat{\epsilon}_{\mathcal{A}} (\epsilon) \Rightarrow  V_{\mathcal{P}} (x) \leq  \epsilon, \; \forall \mathcal{P} \in \mathcal{A}.$$
\end{lemma} 

\begin{proof}
From \eqref{eq:PRL:Def} we have 
$$ \sP_{\hat{\mathcal{P}}}\{E\} \leq \hat{\epsilon}_{\mathcal{A}} (\epsilon)  
        \Rightarrow \sup_{\mathcal{P}\in \mathcal{A}}\sP_{\mathcal{P}}\{E\} \leq \epsilon, \forall E \subseteq \Delta. $$
        Particularizing $E$ to $\{\delta\in \Delta\,:\,g_\delta(x) >0\}$,
        we obtain that $ \sP_{\hat{\mathcal{P}}}\{\delta\in \Delta : g_\delta(x) >0\} \leq \hat{\epsilon}_{\mathcal{A}} (\epsilon)$ implies   
        $$ \sup_{\mathcal{P}\in \mathcal{A}}\sP_{\mathcal{P}}\{\delta\in \Delta: g_\delta(x) >0\} \leq \epsilon. $$
        Equivalently, 
        $  V_{\hat{\mathcal{P}}} (x) \leq \hat{\epsilon}_{\mathcal{A}} (\epsilon) \Rightarrow  V_{\mathcal{P}} (x) \leq  \epsilon, \; \forall \mathcal{P} \in \mathcal{A}.$ \end{proof}%
As a direct consequence of Lemma~\ref{lemma:PRL:VN}, Problem~\eqref{eq:DCCP} can be conservatively approximated as
\begin{subequations} \label{eq:CCP_with_PRL}
\begin{align}
    \min_{\vx\in\sX} \,& \ J(\vx)\\
    \text{s.t.}\quad & \sP_{\hat{\mathcal{P}}} \left(g_{\delta}(\vx)\leq 0\right) \geq 1-\hat{\epsilon}_{\mathcal{A}}(\epsilon).
\end{align}
\end{subequations}
 
Under Assumption~\ref{ass:greater_zero}, the works~\cite{tsengRandomConvexApproximations2016a, jiangDatadrivenChanceConstrained2016} give PRLs for the total variation distance, the Hellinger distance, the Kullback-Leibler distance and the Neyman $\chi^2$ distance.
All four of the mentioned distances belong to the class of $\phi$-divergence-based distances~\cite{jiangDatadrivenChanceConstrained2016}.

In~\cite{nemirovskiConvexApproximationsChance2007,hongAmbiguousProbabilisticPrograms2013}, the computation for the PRL is presented for the following discrepancy functional
    \begin{equation} \label{eq:amb_distances}
    \rho_{\text{RVD}}(\mathcal{P},\hat{\mathcal{P}})=\min_{E\subseteq \Delta_{\mathcal{\hat{P}}}} \{M_{\text{RVD}}: \sP_{\mathcal{P}}(E)\leq M_{\text{RVD}} \sP_{\hat{\mathcal{P}}}(E) \},
\end{equation}
We term this distance, as to our knowledge it has remained unnamed, \emph{Relative Variation Distance (RVD)} due to its similarity to the total variation distance.

The RVD can be written in a differential form for continuous distributions~\cite{hongAmbiguousProbabilisticPrograms2013}
\begin{equation} \label{eq:RVD_differential}
    \rho_{\text{RVD}}(\mathcal{P},\hat{\mathcal{P}})=\sup_{\delta\in\Delta_{\mathcal{\hat{P}}}} \frac{f_{\mathcal{P}}(\delta)}{f_{\hat{\mathcal{P}}}(\delta)}.
\end{equation}
The RVD describes an upper bound on the likelihood ratio of any distribution in the ambiguity set to the nominal distribution.
The RVD can be used to relate the different $\phi$-divergence-based distances with each other~\cite{jiangDatadrivenChanceConstrained2016}.
For the RVD to be finite, Assumption~\ref{ass:greater_zero} needs to hold.
The PRL for the RVD for an ambiguity set bounded by $M_{\text{RVD}}$ was shown in~\cite{nemirovskiConvexApproximationsChance2007} to be
\begin{equation} \label{eq:PRL_RVD}
    \hat{\epsilon}_{M_{\text{RVD}}} (\epsilon)=\frac{\epsilon}{M_{\text{RVD}}}.
\end{equation}
As long as $\epsilon$ is greater than zero and $M_{\text{RVD}}$ is finite, there will always exists a PRL greater than zero.
As a safe design often requires small risk levels, this property is very important.

For many distributions, the RVD can be calculated in a closed form. 
The extreme value theorem together with Assumption~\ref{ass:greater_zero} guarantees a finite $M_{\text{RVD}}$ between two distributions with bounded support, for example uniform or triangular distributions.
Assumption~\ref{ass:greater_zero} also covers distributions with finite support. However for the RVD to exist, the ambiguity set is constrained to the discrete support of the nominal distribution~\cite{hongAmbiguousProbabilisticPrograms2013}.
For distributions with (semi)-infinite support, e.g. the exponential distribution or the normal distribution, the nominal distribution needs to represent the distribution with the longest tail.

One of the contributions of this paper is a closed-form solution for multivariable Gaussian distributions.
\begin{proposition} \label{prop:normal}
    The RVD between two multivariate Gaussian distributions with $\delta\in\sR^{n_d}$
    \begin{align*}
        f_{\mathcal{P}}(\delta)=\frac{1}{\sqrt{(2 \pi)^{n_d}\det{\Sigma}}} \exp{-\frac{1}{2}\left((\delta -\mu)^{\intercal}\Sigma^{-1}(\delta -\mu)\right)},\\
        f_{\hat{\mathcal{P}}}(\delta)=\frac{1}{\sqrt{(2 \pi)^{n_d}\det{\hat{\Sigma}}}} \exp{-\frac{1}{2}\left((\delta -\hat{\mu})^{\intercal}\hat{\Sigma}^{-1}(\delta -\hat{\mu})\right)},
\end{align*}
exists, if $\hat{\Sigma}^{-1}\preceq\Sigma^{-1}$ with
\begin{multline} \label{eq:M_RVD_Gaussian}
    M_{\text{RVD}}=\sqrt{\frac{\det{\hat{\Sigma}}}{\det{\Sigma}} }\exp (-\frac{1}{2}((\delta_{\text{max}} -\mu)^{\intercal}\Sigma^{-1}(\delta_{\text{max}} -\mu) \\
    -(\delta_{\text{max}} -\hat{\mu})^{\intercal}\hat{\Sigma}^{-1}(\delta_{\text{max}} -\hat{\mu}))),
\end{multline}
where 
$    \delta_{\text{max}}=\left[\Sigma^{-1}- \hat{\Sigma}^{-1} \right]^{-1} \left[ \Sigma^{-1} \mu -\hat{\Sigma}^{-1}\hat{\mu}\right]$.
\end{proposition}
\begin{proof}
The RVD can be calculated as the maximum of the ratio $f_{\mathcal{P}}(\delta)/f_{\hat{\mathcal{P}}}(\delta)$. For this, we set the gradient of the ratio to zero, which gives $\delta_{\text{max}}$. The Hessian of the ratio is positive semi-definite at $\delta_{\text{max}}$ if $\hat{\Sigma}^{-1}\preceq\Sigma^{-1}$. For a more detailed proof, we refer to the supplementary material\Footnote{\footnote{\label{arxiv}\url{https://doi.org/10.48550/arXiv.2409.01177}}}\PROOFS{in Section~\ref{sec:supplementary}}.
\OLD{The RVD can be calculated as the maximum of
\begin{multline*}
    q(\delta)=\sqrt{\frac{\det{\hat{\Sigma}}}{\det{\Sigma}} }\exp (-\frac{1}{2}((\delta -\mu)^{\intercal}\Sigma^{-1}(\delta -\mu) \\
    -(\delta -\hat{\mu})^{\intercal}\hat{\Sigma}^{-1}(\delta -\hat{\mu}))),
\end{multline*}
as in~\eqref{eq:RVD_differential}.
For this, we take the first derivative of $q(\delta)$ with respect to $\delta$
\begin{equation} \label{eq:nabla_q_RVD_gauss}
    \nabla_{\delta}  q (\delta)=\left(\hat{\Sigma}^{-1}(\delta-\hat{\mu})-\Sigma^{-1}(\delta-\mu) \right) q(\delta),
\end{equation}
which is zero for
    \begin{align*}
        \delta_{\text{max}}=\left[\hat{\Sigma}^{-1} - \Sigma^{-1} \right]^{-1} \left[\hat{\Sigma}^{-1}\hat{\mu} - \Sigma^{-1} \mu \right],
    \end{align*}
because the first factor of~\eqref{eq:nabla_q_RVD_gauss} is zero.
To verify that $\delta_{\text{max}}$ is indeed a maximum, the hessian $\nabla^2{\delta} q$ is calculated
\begin{multline*}
    \nabla^2_{\delta} q(\delta)= \\
    \left(\hat{\Sigma}^{-1} -\Sigma^{-1}
    +\left(\hat{\Sigma}^{-1}(\delta-\hat{\mu})-\Sigma^{-1}(\delta-\mu)\right)^2\right) q(\delta)
\end{multline*}
If $\delta=\delta_{\text{max}}$ is plugged in, the bracket which was zero for the first derivative and appears squared in $\nabla^2 q$ will also be zero. So the only thing relevant for $\nabla^2 q(\delta_{\text{max}})\preceq0$ is
\begin{align*}
    \Sigma^{-1}-\hat{\Sigma}^{-1}\succ 0.
\end{align*}
Therefore $\hat{\Sigma}^{-1}\preceq\Sigma^{-1}$ is sufficient for the existence of a finite $M_{\text{RVD}}$ for two multivariate normal distributions.
The corresponding RVD, as shown in~\eqref{eq:M_RVD_Gaussian} can be calculated as $M_{\text{RVD}}=q(\delta_{\text{max}})$.}
\end{proof}
The expression simplifies to
$  M_{\text{RVD}}=\sqrt{\frac{\det{\hat{\Sigma}}}{\det{\Sigma}}}$, if $\mu=\hat{\mu}$.\\
The Gaussian distributions belonging to ambiguity sets bounded by different values for the RVD are displayed in Figure~\ref{fig:Gaussian_RVD}. One can see that their upper bound is the nominal density multiplied by $M_{\text{RVD}}$.
\begin{figure}
    \centering
    \begin{subfigure}{.23\textwidth}
        \centering
        \includegraphics[width=1\textwidth]{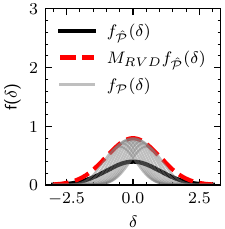}
        \caption{$\mathcal{A}=\rho_{\text{RVD}}(\mathcal{P},\hat{\mathcal{P}})\leq 2$}
        \label{fig:Gaussian_RVD_2}
    \end{subfigure}
        \begin{subfigure}{.23\textwidth}
        \centering
        \includegraphics[width=1\textwidth]{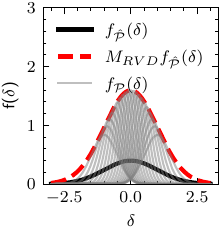}
        \caption{$\mathcal{A}=\rho_{\text{RVD}}(\mathcal{P},\hat{\mathcal{P}})\leq 4$}
        \label{fig:Gaussian_RVD_4}
    \end{subfigure}
    \caption{Density $f_{\mathcal{P}}(\delta)$ of all Gaussian distributions (grey) belonging to ambiguity sets with different sizes for a nominal distribution (black) $f_{\hat{\mathcal{P}}}(\delta)$ with $\hat{\mu}=0,\; \hat{\sigma} =1$. }
    \label{fig:Gaussian_RVD}
    \vspace{-2ex}
\end{figure}
\section{Comparison of the perturbed risk levels} \label{sec:Comp_of_PRL}

For many applications, low risk levels are required. For ambiguity sets described by the RVD distance, the relation  $\hat{\epsilon}_{M_{\text{RVD}}} /\epsilon=  {M_{\text{RVD}}}^{-1}$ is constant due to the alignment of the definitions of the PRL and the RVD.
For the distances for which a PRL is derived in~\cite{tsengRandomConvexApproximations2016a,jiangDatadrivenChanceConstrained2016} the ratio $\hat{\epsilon}_M(\epsilon)/ \epsilon$, tends analytically to zero or can not be used as the desired risk level $\epsilon$ tends to zero.

In the following, we compare the RVD to the distances in~\cite{tsengRandomConvexApproximations2016a,jiangDatadrivenChanceConstrained2016} for a fixed ambiguity in a finite set of normal distributions. 
We focus on PRLs, as they ensure distributional robustness via a straightforward reformulation and therefore, we only consider these discrepancy functionals and omit other discrepancy functionals that are well suited for normal distributions such as the Gelbrich distance.

We assume a discrete ambiguity set of 25 normal distributions $\mathcal{N}(\mu_i,\sigma_i)$ with means and standard deviations uniformly distributed on $\mu_i\in\left[-1,1\right],\ \sigma_i \in \left[1,2\right]$. These discrete distributions could be the result of a bootstrapping procedure to determine the ambiguity set from data.
For each of the discrepancy functionals presented in this section, we find the nominal distribution minimizing the size of the ambiguity set to encompass all of the 25 normal distributions by solving
\begin{subequations} \label{eq:Discr_func_mini_opt}
\begin{align} 
    \min_{\hat{\mu},\hat{\sigma},M} \,& M\\
    \text{s.t.}\quad & \rho(\mathcal{N}(\mu_i,\sigma_i),\mathcal{N}(\hat{\mu},\hat{\sigma}))\leq M,\ \forall i=1,...,25.
\end{align}
\end{subequations}
The code is openly available\footnote{\label{URL} \url{https://github.com/MoritzHein/DistriRobRiskLev}}.
For the RVD, the Hellinger distance as well as the Kullback-Leibler distance, closed form expressions for normal distributions exist~\cite{pardoStatisticalInferenceBased2006}. For the other discrepancy functionals, the integrals defining the functionals are approximated via the trapezoidal rule.
The resulting nominal distributions and the respective size of the ambiguity sets are displayed in Table~\ref{tab:Nominal_Distributions}.
\begin{table}
    \centering
        \begin{tabular}{lllll} \toprule
         Discrepancy functional& $\hat{\mu}$ & $\hat{\sigma}$ & $M$ & $\hat{\epsilon}_M$ for  $\epsilon=0.01$  \\
         \midrule
         RVD& 0.15 & 2.03 & 2.05 & 0.0049\\
         Kullback-Leibler & 0.00 & 1.59 & 0.19 & 0\\
         Hellinger & -0.02 & 1.48 & 0.23 & 0 \\
         $\chi^2$ & 0.05 & 1.75 & 0.36 & 0.0003\\
         Total variation & -0.04 & 1.59 & 0.24 & 0\\
         \bottomrule
    \end{tabular}
    \caption{Nominal distributions for different discrepancy functionals, the size of the respective ambiguity set to cover 25 normal distributions and the respective PRL for $\epsilon=0.01$.}
    \label{tab:Nominal_Distributions}
    \vspace{-3ex}
\end{table}
While for most of the distances the nominal distribution was optimized to be rather near to $\hat{\mu}=0$ and $\hat{\sigma} \approx 1.5$, for the RVD $\hat{\sigma}>2$ due to the necessity of $\hat{\sigma}^{-1}\preceq\sigma^{-1}$, as derived in Proposition~\ref{prop:normal}. Thus, the RVD is sensitive to the distributions with large $\sigma$, which also explains the larger deviation of $\hat{\mu}$ from zero.

Although it may seem restricting that the RVD only contains distributions with a smaller tail,
the nominal distribution is a design parameter for the ambiguity set and can be chosen in advance, similarly to moment-based ambiguity sets~\cite{liDistributionallyRobustOptimization2022}.

For a safety level of $99\%$, which is a typical value for example in the control community, the corresponding PRLs are also listed in Table~\ref{tab:Nominal_Distributions}. For the RVD the corresponding risk level is roughly half the nominal risk level, while for the $\chi^2$-distance, it is around $3\%$ of the nominal risk level. For all other distances the PRL is $0$ and the corresponding distributionally robust chance constraints do not allow any probability of violation.

Figure~\ref{fig:Comp_gaussian} shows the ratio of the PRL $\hat{\epsilon}_M$ to the nominal risk level $\epsilon$. For large $\epsilon$, the PRL of all distances exceeds that of the RVD, but the cost of reducing the nominal risk grows for all functionals except the RVD. For the RVD, the ratio remains constant, so for $\epsilon<0.4$ it outperforms the others.

The better behavior of the RVD for low risks is due to the alignment of the definitions of the PRL in~\eqref{eq:PRL:Def} and of the RVD in~\eqref{eq:amb_distances} resulting in a PRL for the RVD as defined in~\eqref{eq:PRL_RVD}. Examining~\eqref{eq:PRL_RVD}, it is clear that the ratio of risk level and PRL is a constant, in contrast to other commonly used distances. 
\begin{figure}
    \centering
    \includegraphics[width=0.45\textwidth]{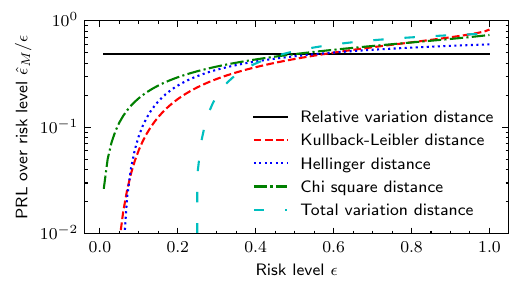}
    \caption{Comparison of the PRL of different discrepancy functionals over $\epsilon$ for an ambiguity set of 25 normal distributions.}
    \label{fig:Comp_gaussian}
    \vspace{-2ex}
\end{figure}

\section{PRLs in randomized methods} \label{sec:application}
This section presents how randomized approaches can be adapted to solve chance constrained problems in a distributionally robust fashion with a PRL.

Suppose that for every multi-sample $\sD_N=\{\delta_1,...,\delta_{N}\} \in \Delta^N,$ where $\delta_i$ are independent and identically distributed samples taken from the distribution $\hat{\mathcal{P}}$, there is an operator providing a value for the decision variable $x\in \mathbb{X}$, which we denote by $\hat{x}_s(\sD_N)$.  

Often $\hat{x}_s(\sD_N)$ is obtained from a scenario optimization problem in which the samples from $\sD_N$ are used to reduce the probability of the event $\{\delta\in \Delta : g_\delta(\hat{x}_s(\sD_N))>0\}$: 
\begin{subequations} \label{eq:RSP_for_SA}
\begin{align}
  \hat{x}_s(\sD_N) = \arg & \min_{\vx\in\sX} \, J(\vx) \\
    \text{s.t.}\quad &  g_{\delta}(\vx)\leq 0 ,\forall \delta \in \sD_N=\{ \delta^1,...,\delta^{N} \}.  
\end{align}
\end{subequations}

Under some assumptions on the functions $g_\delta(\cdot)$ and $J(\cdot)$, like convexity, finite VC-dimension, finite families, Lipschitz continuity, etc., one can derive results that bound the probability that the sampled constraints in $\sD_N$ yield a solution $\hat{\vx}_s(\sD_N)$ that does not satisfy a constraint on the probability of violation (for the nominal distribution $\hat{\mathcal{P}}$)~\cite{campiIntroductionScenarioApproach2018, tempoRandomizedAlgorithmsAnalysis2013, calafioreScenarioApproachRobust2006}. That is, there are numerous frameworks for which it is possible to obtain a function $F_N(\epsilon)$ that satisfies
 $$ \sP_{\hat{\mathcal{P}}}^{N}\{V_{\hat{\mathcal{P}}}(\hat{x}_s(\sD_N)) > \epsilon \}\leq F_N(\epsilon), \; \forall \epsilon \in [0,1].  $$
 We note that the previous inequality applies only to the nominal distribution $\hat{\mathcal{P}}$, since $\hat{x}_s(\sD_N)$ is obtained from samples from $\hat{\mathcal{P}}$. In order to bound the probability that $V_{\mathcal{P}}(\hat{x}_s(\sD_N)) > \epsilon$ in the ambiguity set $\mathcal{A}$, we introduce following assumption.

\begin{assumption}\label{assumption:F}
It is assumed that \blista
\item $\hat{\epsilon}_{\mathcal{A}}: [0,1]\to[0,1]$  is a  PRL for the ambiguity set $\mathcal{A}$ and the nominal distribution $\hat{\mathcal{P}}\in \mathcal{A}$. %

\item The operator $\hat{x}_s:\Delta^N \to \mathbb{X}$ satisfies 
    $$ \sP_{\hat{\mathcal{P}}}^{N}\{V_{\hat{\mathcal{P}}}(\hat{x}_s(\sD_N)) > \epsilon \}\leq F_N(\epsilon), \; \forall \epsilon\in [0,1],  $$
    where $\sD_N=\{\delta_1,...,\delta_{N}\}$ denotes $N$  independent and identically distributed samples drawn from the nominal distribution $\hat{\mathcal{P}}$.
\elista
\end{assumption}
We are now in a position to provide a distributional robust result on the probability of violation of $\hat{\vx}_s(\sD_N)$ in the ambiguity set $\mathcal{A}$, which generalizes the result from~\cite{tsengRandomConvexApproximations2016a}.
\begin{lemma} \label{lemma:Bound:PN:with:FN}
Under Assumptions~\ref{ass:greater_zero} and~\ref{assumption:F} we have
 \begin{equation}\label{equ:FN:in:Lemma}
  \sP_{\hat{\mathcal{P}}}^{N} \{ V_{\mathcal{P}}(\hat{\vx}_s(\sD_N))>\epsilon  \}\leq F_N(\hat{\epsilon}_\mathcal{A}(\epsilon)), \; \forall \mathcal{P} \in \mathcal{A}.
    \end{equation}
\end{lemma}
\begin{proof}
From the first point of Assumption \ref{assumption:F}, we have that $\hat{\epsilon}_{\mathcal{A}}: [0,1]\to[0,1]$  is a  PRL for the ambiguity set $\mathcal{A}$ and the nominal distribution $\hat{\mathcal{P}}\in \mathcal{A}$. Thus, we obtain by a direct application of Lemma \ref{lemma:PRL:VN} that   $$ V_{\hat{\mathcal{P}}} (\hat{x}_s(\sD_N)) \leq \hat{\epsilon}_{\mathcal{A}} (\epsilon) \Rightarrow  V_{\mathcal{P}} (\hat{x}_s(\sD_N)) \leq  \epsilon, \; \forall \mathcal{P} \in \mathcal{A}. $$
This implies that, for every $\mathcal{P}\in \mathcal{A}$, 
 \begin{gather}
 \sP_{\hat{\mathcal{P}}}^{N}\{V_{\hat{\mathcal{P}}}(\hat{x}_s(\sD_N)) \leq  \hat{\epsilon}_{\mathcal{A}}(\epsilon) \}  \leq  \sP_{\hat{\mathcal{P}}}^{N} \{ V_{\mathcal{P}}(\hat{\vx}_s(\sD_N))\leq \epsilon  \}, \\
\label{ineq:sup:in:A} \sP_{\hat{\mathcal{P}}}^{N} \{ V_{\mathcal{P}}(\hat{\vx}_s(\sD_N)) >\epsilon  \} \leq \sP_{\hat{\mathcal{P}}}^{N}\{V_{\hat{\mathcal{P}}}(\hat{x}_s(\sD_N)) >  \hat{\epsilon}_{\mathcal{A}}(\epsilon) \} . 
\end{gather}
 From the second point of Assumption \ref{assumption:F} we have 
$$ \sP_{\hat{\mathcal{P}}}^{N}\{V_{\hat{\mathcal{P}}}(\hat{x}_s(\sD_N)) > \hat{\epsilon}_{\mathcal{A}}(\epsilon) \}\leq F_N(\hat{\epsilon}_{\mathcal{A}}(\epsilon) ). $$ 
This, along with \eqref{ineq:sup:in:A}, proves the claim of the lemma.
\end{proof}

Lemma~\ref{lemma:Bound:PN:with:FN} gives distributionally robust two-level probabilistic bounds on the violation probability with the desired risk level $\epsilon$ and the confidence $1-F_N(\hat{\epsilon}_{\mathcal{A}}(\epsilon) )$. In the following, we will present distributionally robust one-level probability results corresponding to the expected probability of violation.

\begin{definition}[Expected Probability of Violation] \label{def:Expected_violation_Prob}
Given the integer $N\geq 1$ and the operator $\hat{x}_s:\Delta ^N \to \mathbb{X}$, denote by $\E^N_{\hat{\mathcal{P}}}\{ V_\mathcal{P} (\hat{\vx}_s(\sD_N))\}$ the mean value of the random variable $V_{\mathcal{P}} (\hat{\vx}_s(\sD_N))$, where $\sD_N\in \Delta^N $ denotes an i.i.d. multi-sample drawn from the nominal distribution $\hat{\mathcal{P}}$. 
\end{definition}
The expected probability of violation corresponds to the violation probability of the solution $\hat{\vx}_s(\sD_N)$ on a new sample $\delta^{N+1}$ and gives a more intuitive understanding of the actual safety than the nested two-level guarantees~\cite{campiNotesScenarioDesign2009}.
\begin{lemma} \label{lemma:Bound:Expected:VN}
Under Assumptions~\ref{ass:greater_zero} and~\ref{assumption:F} we have 
$$\E^N_{\hat{\mathcal{P}}}\{ V_\mathcal{P} (\hat{\vx}_s(\sD_N))\} \leq \int_0^1 F_N (\hat{\epsilon}_\mathcal{A}(\epsilon))d\epsilon, \; \forall \mathcal{P} \in \mathcal{A}.$$
\end{lemma}
\begin{proof}
    Denote by $f_{\mathcal{P}}(\epsilon)$ the density function of the random variable $V_\mathcal{P}(\hat{x}_s (\sD_N))$, where $\sD_N$ is drawn according to $\hat{\mathcal{P}} \in \mathcal{A}$. By definition, \begin{eqnarray*} \label{eq:density_definition}
    \int_0^\epsilon f_\mathcal{P} (\tau) d\tau & = & \sP_{\hat{\mathcal{P}}}^{N} \{ V_{\mathcal{P}}(\hat{\vx}_s(\sD_N)) \leq \epsilon  \} \\
    & = & 1 - \sP_{\hat{\mathcal{P}}}^{N} \{ V_{\mathcal{P}}(\hat{\vx}_s(\sD_N)) > \epsilon  \} 
    \end{eqnarray*}
    From Lemma \ref{lemma:Bound:PN:with:FN} we have $ \sP_{\hat{\mathcal{P}}}^{N} \{ V_{\mathcal{P}}(\hat{\vx}_s(\sD_N)) > \epsilon\} \leq F_N (\hat{\epsilon}_\mathcal{A}(\epsilon)) $ $\forall \gP \in \gA$. Thus, 
    $$ 
     \int_0^\epsilon f_\mathcal{P} (\tau) d\tau  
     \geq  1 - F_N (\hat{\epsilon}_\mathcal{A}(\epsilon)),\;  \forall \epsilon \in [0,1].$$
     Integrating this inequality and changing the order of integration via Fubini's theorem  gives
    \begin{align*}
        &\int_0^1 \left(\int_0^\epsilon f_\mathcal{P} (\tau) d\tau\right) d\epsilon  
     \geq  1 - \int_0^1 F_N (\hat{\epsilon}_\mathcal{A}(\epsilon))d\epsilon \\
     \Leftrightarrow & \int_0^1 \left(\int_0^1 I(\epsilon-\tau)f_\mathcal{P} (\tau) d\tau\right) d\epsilon  
     \geq  1 - \int_0^1 F_N (\hat{\epsilon}_\mathcal{A}(\epsilon))d\epsilon\\    
\Leftrightarrow& \int_0^1 \left(\int_0^1 I(\epsilon-\tau)f_\mathcal{P} (\tau) d\epsilon\right) d\tau  
     \geq  1 - \int_0^1 F_N (\hat{\epsilon}_\mathcal{A}(\epsilon))d\epsilon.
\end{align*}    
We note that 
$$ \int_0^1 I(\epsilon-\tau)f_\mathcal{P} (\tau) d\epsilon  =  f_\mathcal{P} (\tau)  \int_0^1 I(\epsilon-\tau)d\epsilon = f_\mathcal{P} (\tau)  (1- \tau).$$  
Thus, we can rewrite the inequality as 
$$ \int_0^1 (1-\tau) f_\mathcal{P} (\tau) d\tau  
     \geq  1 - \int_0^1 F_N (\hat{\epsilon}_\mathcal{A}(\epsilon))d\epsilon.$$
Rearranging the terms of the obtained expression we obtain
$ \E^N_{\hat{\mathcal{P}}}\{ V_\mathcal{P} (\hat{\vx}_s(\sD_N))\} =  \int_0^1 \tau f_\mathcal{P} (\tau) d\tau   \leq \int_0^1 F_N (\hat{\epsilon}_\mathcal{A}(\epsilon))d\epsilon.$
\end{proof}

\begin{corollary} \label{Corollary:Bound:Expected:ByParts}
Suppose that Assumptions~\ref{ass:greater_zero} and~\ref{assumption:F} hold and that $1-F_N (\hat{\epsilon}_\mathcal{A}(\epsilon))$ has a continuous derivative with respect to $\epsilon$ that we denote $\hat{f}_\mathcal{A}(\cdot)$. 
Then, 
$$\E^N_{\hat{\mathcal{P}}}\{ V_\mathcal{P} (\hat{\vx}_s(\sD_N))\} \leq F_N(\hat{\epsilon}_\mathcal{A}(1)) +  \int_0^1 \epsilon \hat{f}_\mathcal{A}(\epsilon) d\epsilon.$$
\end{corollary}
\begin{proof}
Since both $\epsilon$ and $F_N(\hat{\epsilon}_\mathcal{A}(\epsilon))$ have continuous derivatives with respect to $\epsilon$,  integration by parts gives
\begin{eqnarray*} \int_0^1 F_N(\hat{\epsilon}_\mathcal{A}(\epsilon))d\epsilon &=&  \left[ \epsilon F_N(\hat{\epsilon}_\mathcal{A}(\epsilon)) \right]_0^1 -\int_0^1 \epsilon \frac{d}{d\epsilon}F_N(\hat{\epsilon}_\mathcal{A}(\epsilon)) d\epsilon \\
&=&  F_N(\hat{\epsilon}_\mathcal{A}(1)) +  \int_0^1 \epsilon \hat{f}_\mathcal{A}(\epsilon) d\epsilon. 
\end{eqnarray*}
The claim now follows from the application of Lemma~\ref{lemma:Bound:Expected:VN}.
\end{proof}
Corollary~\ref{Corollary:Bound:Expected:ByParts} recovers the interpretation of the expected probability of violation as the first moment of $V_\mathcal{P} (\hat{\vx}_s(\sD_N))$.

\OLD{\subsection{Distributionally robust convex chance constraint with the scenario approach}
One framework to give probabilistic guarantees on the violation probability in accordance with Assumption~\ref{assumption:F}(ii) is the scenario approach~\cite{calafioreScenarioApproachRobust2006,campiScenarioApproachSystems2009,campiIntroductionScenarioApproach2018}.
The scenario approach gives probabilistic guarantees on the chance constraint violation probability of the optimal solution $\hat{\vx}_{SA}(\sD_N)$ of~\eqref{eq:RSP_for_SA}, the scenario counterpart of~\eqref{eq:CCP}, for convex constraints and cost function.
Depending on $N$, the probability $V_{\hat{\mathcal{P}}}(\hat{\vx}_{SA}(\sD_N))=\sP_{\hat{\mathcal{P}}}\{\delta\in\Delta: g_{\delta}(\hat{\vx}_{SA}(\sD_N)))> 0\}$ 
can be upper bounded by $\epsilon$ with a certainty of $1-\beta=1-F_N^{\text{SA}}(\epsilon)$.
This result assumes the existence and uniqueness of the solution~\cite{campiIntroductionScenarioApproach2018}.
\begin{assumption} \label{ass:scen_approach}
    For every sample size $N\geq \dim{\vx}$ and for every sample $\sD_N=\{\delta_1,...,\delta_{N}\}$ the solution of program~\eqref{eq:RSP_for_SA} exists and is unique. Also, the program~\eqref{eq:RSP_for_SA} is convex in $\vx$. 
\end{assumption}
The scenario approach (see Theorem 3.7 in~\cite{campiIntroductionScenarioApproach2018} and Theorem 1 in~\cite{campiExactFeasibilityRandomized2008}) assures these bounds, when $N$ is chosen to satisfy
\begin{equation} \label{eq:bound_scen}
    \sP_{\hat{\mathcal{P}}}^{N}\{V_{\hat{\mathcal{P}}}(\hat{\vx}_{SA}(\sD_N)) > \epsilon \}\leq \sum_{i=0}^{d-1} f_B(N,i,\epsilon) =F_N^{\text{SA}}(\epsilon)= \beta,
\end{equation}
where $d=\dim{\vx}$
and $f_B(N,l,\alpha)$ denotes the probability mass function of the binomial distribution for $l$ successes in $N$ independent trials with the rate of success $\alpha$
\begin{align} \label{eq:binom}
    f_B(N,l,\alpha)=\binom{N}{l}\alpha^l(1-\alpha)^{N-l}.
\end{align}
For the proof and for more information about the scenario approach, see~\cite{campiIntroductionScenarioApproach2018}.
The bound in~\eqref{eq:bound_scen} shows that the violation probability $V_{\mathcal{P}}(\hat{\vx}_{SA}(\sD_N))=\sP\{\delta\in\Delta : g_{\delta}(\hat{\vx}_{SA}(\sD_N))> 0\}$ is upper bounded by a beta distribution with $d$ and $N-d+1$ degrees of freedom with the probability density
\begin{equation} \label{eq:SA_density}
    \frac{\partial (1-F_N^{\text{SA}}(\epsilon))}{\partial \epsilon}=f^{\text{SA}}_{\hat{\mathcal{P}}}(\epsilon)=d\binom{N}{d}\epsilon^{d-1}(1-\epsilon)^{N-d}.
\end{equation}
This probability density function describes the distribution of the actual violation probability of a newly sampled solution $\hat{\vx}_{SA}(\sD_N)$ and is useful to derive one-level probability results~\cite{campiNotesScenarioDesign2009}. 
The bound in~\eqref{eq:bound_scen} can be used to calculate an upper bound on $N$, as done in~\cite{alamoRandomizedMethodsDesign2015}.
One difficulty arising with the scenario approach is that to obtain guarantees for low probabilities of violation, many samples are necessary, making the optimal solution of the sampled program conservative.
To remedy this conservatism, a sampling and discarding approach was proposed in~\cite{campiSamplingandDiscardingApproachChanceConstrained2011, calafioreRandomConvexPrograms2010}.
The derived guarantees are structurally similar to the scenario approach, but include as additional parameter the number of discarded constraints $k$. All of the $k$ discarded constraints are necessitated to be violated by the solution $\hat{\vx}_{SA,k}(\sD_N)$ of the scenario program without the discarded samples. That is, under Assumption~\ref{ass:scen_approach}, one obtains (see e.g. Theorem 3.9 in~\cite{campiIntroductionScenarioApproach2018})
\begin{multline} \label{eq:bound_scen_CR}
    \sP_{\hat{\mathcal{P}}}^{N}\{V_{\hat{\mathcal{P}}}(\hat{\vx}_{SA,k}(\sD_N)) > \epsilon \}\\
    \leq \binom{k+d-1}{k} \sum_{i=0}^{d+k-1} f_B(N, i, \epsilon) = F_N^{\text{SA},k}(\epsilon).
\end{multline}
Following Lemma~\ref{lemma:Bound:PN:with:FN}, the results of the scenario approach can be made distributionally robust by replacing $\epsilon$ in~\eqref{eq:bound_scen} and~\eqref{eq:bound_scen_CR} with the PRL $\hat{\epsilon}_{\mathcal{A}}(\epsilon)$ corresponding to the used discrepancy functional, as it has been already proposed~\cite{tsengRandomConvexApproximations2016a,erdoganAmbiguousChanceConstrained2006, madhusudanaraoLearningAmbiguousChance2023a}
\begin{equation} \label{eq:Distributionally_robust_Bound_SA}
    \sP_{\hat{\mathcal{P}}}^{N} \{\max_{\mathcal{P}\in\mathcal{A}} V_{\mathcal{P}}(\hat{\vx}^*)>\epsilon  \}\\
    \leq\ \sum_{i=0}^{d-1} f_B(N,i,\hat{\epsilon}_{\mathcal{A}})=F_N^{\text{SA}}(\hat{\epsilon}_{\mathcal{A}}(\epsilon)).
\end{equation}
To our knowledge, it has not been discussed in the literature that the application of the PRLs is also possible for the scenario approach with constraint removal, which we present as an additional contribution next.
\begin{result}
    Consider problem~\eqref{eq:RSP_for_SA} under Assumption~\ref{ass:scen_approach}, where the sample set $\sD_N$ has been drawn according to $\hat{\mathcal{P}}$ with $k$ of the $N_{\text{samp}}$ sampled constraints removed, which are all violated by $\hat{\vx}_{SA,k}(\sD_N)$. Assuming that the real uncertainty lies within an ambiguity set $\mathcal{A}$, for which the PRL $\hat{\epsilon}_{\mathcal{A}}$ exists for a given risk level $\epsilon$, and  defining $\beta$ as
    \begin{equation} \label{eq:Distributionally_robust_Bound_CR}
    \beta
    = \binom{k+d-1}{k} \sum_{i=0}^{d+k-1} f_B(N, i, \hat{\epsilon}_{\mathcal{A}}(\epsilon)),
\end{equation}
    one has that the inequality $V_{\mathcal{P}}(\hat{\vx}_{SA,k}(\sD_N))\leq\epsilon$ is satisfied with a probability no smaller than $1-\beta,\ \forall \mathcal{P}\in\mathcal{A}$. In other words,
    \begin{align}
  \sP_{\hat{\mathcal{P}}}^{N} \{\max_{\mathcal{P}\in\mathcal{A}} V_{\mathcal{P}}(\hat{\vx}_k^*)>\epsilon  \}\leq \beta.
    \end{align}
\end{result}
\begin{proof}
    The proof is the straightforward application of Lemma~\ref{lemma:Bound:PN:with:FN} with $\beta=F_N^{\text{SA},k}(\hat{\epsilon}_{\mathcal{A}}(\epsilon).$
\end{proof}
In~\cite{campiNotesScenarioDesign2009}, a one level probability result for the scenario approach was presented.
This one level result corresponds to the 
expected violation probability (see Definition~\ref{def:Expected_violation_Prob}).
The expected violation probability can be upper bounded by the mean of the beta distribution with $d$ and $N-d+1$ degrees of freedom~\cite{campiNotesScenarioDesign2009,campiIntroductionScenarioApproach2018}
\begin{equation} \label{eq:mean_eps_scen_app}
    \E^N_{\hat{\mathcal{P}}}\{ V_{\hat{\mathcal{P}} }(\hat{\vx}_s(\sD_N))\}  \leq \int_0^1 \epsilon f^{\text{SA}}_{\hat{\mathcal{P}}}(\epsilon) d\epsilon = \frac{d}{N+1}.
\end{equation}
For the distributionally robust result, this can be adapted analog to Lemma~\ref{lemma:Bound:Expected:VN}. For the RVD, we show specifically:
\begin{theorem} \label{thm:Expected_RVD}
     Consider problem~\eqref{eq:RSP_for_SA} under assumption~\ref{ass:scen_approach}, where the sample set $\sD_N$ has been drawn according to $\hat{\mathcal{P}}$ with the optimal solution $\hat{\vx}_{SA}(\sD_N)$. Assuming that the real uncertainty lies within an ambiguity set $\mathcal{A}$, bounded by the RVD with radius ${M_{\text{RVD}}}$, the expected probability of violation is
     \begin{multline} \label{eq:thm:expected}
    \E^N_{\hat{\mathcal{P}}}\{ V_\mathcal{P} (\hat{\vx}_s(\sD_N))\} \leq \sum_{i=d}^{N} f_B(N,i, \frac{1}{M_{\text{RVD}}}) \frac{d}{i+1}\\
     +\sum_{i=0}^{d-1} f_B(N,i,M_{\text{RVD}}^{-1}).
     \end{multline}
\end{theorem}
\begin{proof}
    To calculate the worst case mean of $V_{\mathcal{P}}(\hat{\vx}_s(\sD_N))$ for ambiguity measured by the RVD, we use Corollary~\ref{Corollary:Bound:Expected:ByParts}.
    $1-F_N^{\text{SA}} (\hat{\epsilon}_{M_{\text{RVD}}}(\epsilon))$ is continuously differentiable with the corresponding function derived with~\eqref{eq:SA_density} and the chain rule:
    \begin{equation*}
        f^{\text{SA}}_{\text{RVD}}(\epsilon)=\frac{d}{M_{\text{RVD}}}\binom{N}{d}\left( \frac{\epsilon}{M_{\text{RVD}}}\right)^{d-1}(1-\frac{\epsilon}{M_{\text{RVD}}})^{N-d}.
    \end{equation*}
    As also $F_N^{\text{SA}} (\hat{\epsilon}_{M_{\text{RVD}}}(1))=F_N^{\text{SA}}(M_{\text{RVD}}^{-1})$ the following holds:
    \begin{multline*} \label{eq:mean_eps_scen_app_DR_RVD}\max_{\mathcal{P}\in\mathcal{A}}\E_{\mathcal{P}}\left[V_{\mathcal{P}}(\hat{\vx}_s(\sD_N))\right] \leq\int_0^1 \epsilon f^{\text{SA}}_{\text{RVD}}(\epsilon) d\epsilon +\sum_{i=0}^{d-1} f_B(N,i,M_{\text{RVD}}^{-1}) .
    \end{multline*}
Reformulating the integral leads to
\begin{multline}
    \int_0^1 d \binom{N}{d} (\frac{\epsilon}{M_{\text{RVD}}}) ^d \left(1-\frac{\epsilon}{M_{\text{RVD}}} \right)^{N-d} d\epsilon\\
    = \int_0^1  d \binom{N}{d} (\frac{\epsilon}{M_{\text{RVD}}})^d  \left((1-\frac{1}{M_{\text{RVD}}}) +\frac{1}{M_{\text{RVD}}} (1-\eps) \right)^{N-d} d\epsilon \\
    = \int_0^1  \sum_{i=0}^{N-d} d \binom{N}{d} \binom{N-d}{i} \frac{\epsilon^{d}}{M_{\text{RVD}}^{i+d}} (1-\epsilon)^{i}(1-\frac{1}{M_{\text{RVD}}})^{N-d-i} d\epsilon\\
    = \int_0^1 \sum_{i=d}^{N} \binom{N}{i} \frac{1}{M_{\text{RVD}}^i}(1-\frac{1}{M_{\text{RVD}}})^{N-i} \epsilon d\binom{i}{d}\epsilon^{d-1}(1-\epsilon)^{i-d}   d\epsilon\\
    = \sum_{i=d}^{N} \binom{N}{i} \frac{1}{M_{\text{RVD}}^i} (1-\frac{1}{M_{\text{RVD}}})^{N-i} \int_0^1 \epsilon f^{\text{SA}}_{\hat{\mathcal{P}}}(\epsilon)   d\epsilon\\
    =\sum_{i=d}^{N} f_B(N,i, \frac{1}{M_{\text{RVD}}}) \frac{d}{i+1}.
\end{multline}
The second equality is a reformulation of the last bracket, for which the third equality applies the binomial formula. The fourth equality is derived via an index shift, while for the last equalities~\eqref{eq:binom} and~\eqref{eq:mean_eps_scen_app} are used respectively. Using the reformulated integral, we obtain~\eqref{eq:thm:expected}.
\end{proof}
Theorem~\ref{thm:Expected_RVD} gives an analytical solution to the distributionally robust bounds on the expected probability of violation for ambiguity measured by the RVD. Similar to the RVD being able to produce feasible two-level results even for low risk levels, the corresponding one-level bound can be interpreted as a weighted mean of~\eqref{eq:mean_eps_scen_app} for sample sizes $d\leq \hat{N}\leq N$ plus a constant for $\hat{N}<d$.}

\section{Distributionally robust randomized MPC} \label{sec:Case_Study}
In this section, we apply Lemma~\ref{lemma:Bound:PN:with:FN} and~\ref{lemma:Bound:Expected:VN} to the scenario approach~\cite{calafioreScenarioApproachRobust2006,campiScenarioApproachSystems2009,campiIntroductionScenarioApproach2018} in the context of randomized MPC (rMPC)~\cite{schildbachRandomizedModelPredictive2012} under distributional ambiguity described with the RVD.
Under the assumptions of convexity and the existence and uniqueness of a solution of~\eqref{eq:RSP_for_SA}, the scenario approach gives the following formula for $F_N(\eps)$ of Assumption~\ref{assumption:F} (ii)
(see Theorem 3.7 in~\cite{campiIntroductionScenarioApproach2018})
\begin{equation} \label{eq:bound_scen}
    \sP_{\hat{\mathcal{P}}}^{N}\{V_{\hat{\mathcal{P}}}(\hat{\vx}_{SA}(\sD_N)) > \epsilon \}\leq \sum_{i=0}^{d-1} f_B(N,i,\epsilon) =F_N^{\text{SA}}(\epsilon),
\end{equation}
where $d\leq\dim{\vx}$ is the number of support constraints
and $f_B(N,l,\alpha)$ denotes the probability mass function of the binomial distribution 
$    f_B(N,l,\alpha)=\binom{N}{l}\alpha^l(1-\alpha)^{N-l}$.
Formula~\eqref{eq:bound_scen} can be used in Lemma~\ref{lemma:Bound:PN:with:FN} and~\ref{lemma:Bound:Expected:VN} to achieve distributionally robust one- and two-level guarantees for the scenario approach.
The bound in~\eqref{eq:bound_scen} shows that the violation probability $V_{\mathcal{P}}(\hat{\vx}_{SA}(\sD_N))=\sP\{\delta\in\Delta : g_{\delta}(\hat{\vx}_{SA}(\sD_N))> 0\}$ is upper bounded by a beta distribution with $d$ and $N-d+1$ degrees of freedom with the probability density
\begin{equation} \label{eq:SA_density}
    \frac{\partial (1-F_N^{\text{SA}}(\epsilon))}{\partial \epsilon}=f^{\text{SA}}_{\hat{\mathcal{P}}}(\epsilon)=d\binom{N}{d}\epsilon^{d-1}(1-\epsilon)^{N-d}.
\end{equation}

In the context of rMPC, the number of support constraints $d$ is upper bounded by $N_{\text{pred}} n_u$, where $n_u$ is the input dimension and $N_{\text{pred}}$ the prediction horizon of the rMPC~\cite{schildbachRandomizedModelPredictive2012}.

We apply the rMPC for a double-integrator system under linear feedback and time-invariant additive disturbance
\begin{gather*}
    \vx_{k+1}=f_{\mK}(\vx_k,\vc_k)+\vw=(\mA+\mB\mK)\vx_k+\mB \vc_k +\vw,\\
    \mA = \begin{bmatrix}
        1 & 1\\
        0 & 1
    \end{bmatrix}, \mB=\begin{bmatrix}
        0.5\\
        1
    \end{bmatrix}, \mK = \begin{bmatrix}
        -0.43 & -1.03
    \end{bmatrix},
\end{gather*}
with $\vx_k\in \sR^2$, $\vu_k=\mK\vx_k+\vc_k \in\sR$ and we assume a uniform nominal uncertainty $\hat{\vw}\in\sR^2$ with $\hat{w}\sim \hat{\gP}=\gU_{\left[-0.2,0.2\right]^2}$.
The constraints for states $-0.5\leq \vx_{k,i} \leq 2 ,\ i=1,2$ and the input $-1\leq\vu_k\leq1$ are all considered as a joint chance constraint, i.e. the violation probability describes the probability that at least one constraint is violated.
We choose a prediction horizon of $N_{\text{pred}}=2$ with an objective $J(\vx,\vu)=\vx_{\left[0:N_{\text{pred}}+1\right]}^{\intercal}\vx_{\left[0:N_{\text{pred}}+1\right]}+\vc_{\left[0:N_{\text{pred}}\right]}^{\intercal}\vc_{\left[0:N_{\text{pred}}\right]}$ to consider prediction trajectories that are always close to the constraints, leading to a better illustration of the obtained probabilistic guarantees. %

We introduce distributional ambiguity, 
by assuming that the real uncertainty follows a different distribution $\vw\sim \gP$:
\begin{equation*}
    f_{\gP}(\vw)=\begin{cases}
        25 &\text{, if }0.1<|\vw_i|\leq 0.2, \forall i=1,2\\
        0 &\text{, else},
    \end{cases}
\end{equation*}
which is uniformly distributed at the boundaries of the uncertainty support to estimate a worst-case distribution.
An ambiguity set described by the RVD of size $M_{\text{RVD}}=4$ includes this distribution (see the inner plot of Figure~\ref{fig:viol_prob}).
The total variation distance for this example is $M_{\text{TVD}}=0.75$.

To validate the distributionally robust one- and two-level guarantees, we apply the input trajectory of an rMPC with $N=1000$ samples from $\hat{\mathcal{P}}$ in open-loop fashion on 40000 trajectories of length $N_{\text{pred}}$, where $\vw\sim \gP$ to estimate the violation probability of this specific solution.
As the violation probability is a random variable itself, we iterate this process $800$ times to create a histogram of the violation probability shown in Figure~\ref{fig:viol_prob} in blue and compare it to the nominal case without distributional ambiguity (orange).
\begin{figure}
    \centering
    \includegraphics[width=0.45\textwidth]{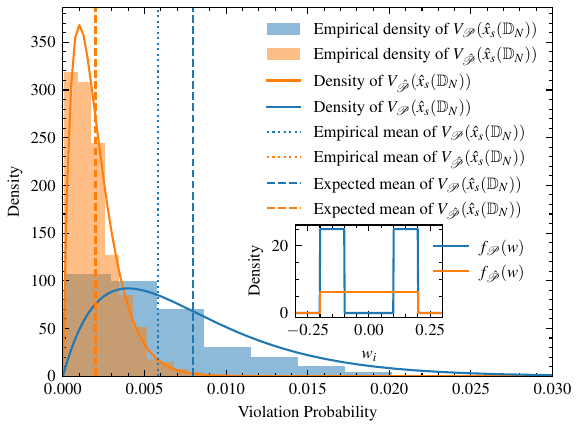}
    \caption{Empirical and analytical density of the violation probability in the rMPC case-study for the case with distributional ambiguity ($V_\mathcal{P} (\hat{\vx}_s(\sD_N))$) and the nominal case ($V_{\hat{\mathcal{P}}} (\hat{\vx}_s(\sD_N))$). The inner plot shows the density functions for the distributions $\hat{\gP}$ and $\gP$.}
    \label{fig:viol_prob}
    \vspace{-3ex}
\end{figure}
The orange density is given by~\eqref{eq:SA_density} and the nominal expected violation probability is given as in~\cite{campiNotesScenarioDesign2009}
\begin{equation} \label{eq:mean_eps_scen_app}
    \E^N_{\hat{\mathcal{P}}}\{ V_{\hat{\mathcal{P}} }(\hat{\vx}_s(\sD_N))\}  \leq \int_0^1 \epsilon f^{\text{SA}}_{\hat{\mathcal{P}}}(\epsilon) d\epsilon = \frac{d}{N+1}.
\end{equation}
Lemma~\ref{lemma:Bound:PN:with:FN} gives the cumulative distribution of the violation probability as $F_N^{\text{SA}}(\hat{\epsilon}_{M_{\text{RVD}}}(\epsilon))$ with~\eqref{eq:bound_scen}.
To calculate the worst-case mean of $V_{\mathcal{P}}(\hat{\vx}_s(\sD_N))$ for ambiguity measured by the RVD, we use Corollary~\ref{Corollary:Bound:Expected:ByParts}.
$1-F_N^{\text{SA}} (\hat{\epsilon}_{M_{\text{RVD}}}(\epsilon))$ is continuously differentiable with the corresponding function for the density derived with~\eqref{eq:SA_density} and the chain rule:
\begin{equation*}
    f^{\text{SA}}_{\text{RVD}}(\epsilon)=\frac{d}{M_{\text{RVD}}}\binom{N}{d}\left( \frac{\epsilon}{M_{\text{RVD}}}\right)^{d-1}(1-\frac{\epsilon}{M_{\text{RVD}}})^{N-d},
\end{equation*}
which gives the blue line in Figure~\ref{fig:viol_prob}.
As also $F_N^{\text{SA}} (\hat{\epsilon}_{M_{\text{RVD}}}(1))=F_N^{\text{SA}}(M_{\text{RVD}}^{-1})$ the following holds by solving the integral in Corollary~\ref{Corollary:Bound:Expected:ByParts} (for details see\Footnote{$^{\text{\ref{arxiv}}}$}\PROOFS{Sec.~\ref{sec:supplementary%
}})
\begin{multline*} %
    \E^N_{\hat{\mathcal{P}}}\{ V_\mathcal{P} (\hat{\vx}_s(\sD_N))\} \leq \sum_{i=d}^{N} f_B(N,i, \frac{1}{M_{\text{RVD}}}) \frac{d}{i+1}\\%
     +\sum_{i=0}^{d-1} f_B(N,i,M_{\text{RVD}}^{-1}),
    \end{multline*}
giving the dotted blue line in Figure~\ref{fig:viol_prob}.
The nominal histogram and empirical mean violation probability match the analytical solution, while the distributionally robust bounds slightly overestimate both. The RVD density closely follows the histogram, whereas the total variation distance would shift the nominal distribution to higher violations by 
$M_{\text{TVD}}$, making the lowest guaranteeable probability 0.75. This highlights the accuracy of the PRL with the RVD even at low risk levels.
\vspace{-1ex}
\section{Conclusion} \label{sec:Conclusion}
We propose the relative variation distance as a means to define perturbed risk levels, enabling consistent rescaling of nominal risk levels even at low violation probabilities where other discrepancy measures are overly conservative. We established theoretical guarantees for randomized solutions of chance constraints and validated them in a model predictive control case study. These results demonstrate the relative variation distance as an effective and practical tool for distributionally robust design in safety-critical settings.
\bibliographystyle{IEEEtran}
\bibliography{IEEEabrv,PhD_small_letters}

\PROOFS{
\newpage
\section{Supplementary material} \label{sec:supplementary}
\subsection{Proof of Proposition~\ref{prop:normal}}
\begin{proof}
    The RVD can be calculated as the maximum of
\begin{multline*}
    q(\delta)=\sqrt{\frac{\det{\hat{\Sigma}}}{\det{\Sigma}} }\exp (-\frac{1}{2}((\delta -\mu)^{\intercal}\Sigma^{-1}(\delta -\mu) \\
    -(\delta -\hat{\mu})^{\intercal}\hat{\Sigma}^{-1}(\delta -\hat{\mu}))),
\end{multline*}
as in~\eqref{eq:RVD_differential}.
For this, we take the first derivative of $q(\delta)$ with respect to $\delta$
\begin{equation} \label{eq:nabla_q_RVD_gauss}
    \nabla_{\delta}  q (\delta)=\left(\hat{\Sigma}^{-1}(\delta-\hat{\mu})-\Sigma^{-1}(\delta-\mu) \right) q(\delta),
\end{equation}
which is zero for
    \begin{align*}
        \delta_{\text{max}}=\left[\hat{\Sigma}^{-1} - \Sigma^{-1} \right]^{-1} \left[\hat{\Sigma}^{-1}\hat{\mu} - \Sigma^{-1} \mu \right],
    \end{align*}
because the first factor of~\eqref{eq:nabla_q_RVD_gauss} is zero.
To verify that $\delta_{\text{max}}$ is indeed a maximum, the hessian $\nabla^2{\delta} q$ is calculated
\begin{multline*}
    \nabla^2_{\delta} q(\delta)= \\
    \left(\hat{\Sigma}^{-1} -\Sigma^{-1}
    +\left(\hat{\Sigma}^{-1}(\delta-\hat{\mu})-\Sigma^{-1}(\delta-\mu)\right)^2\right) q(\delta)
\end{multline*}
If $\delta=\delta_{\text{max}}$ is plugged in, the bracket which was zero for the first derivative and appears squared in $\nabla^2 q$ will also be zero. So the only thing relevant for $\nabla^2 q(\delta_{\text{max}})\preceq0$ is
\begin{align*}
    \Sigma^{-1}-\hat{\Sigma}^{-1}\succ 0.
\end{align*}
Therefore $\hat{\Sigma}^{-1}\preceq\Sigma^{-1}$ is sufficient for the existence of a finite $M_{\text{RVD}}$ for two multivariate normal distributions.
The corresponding RVD, as shown in~\eqref{eq:M_RVD_Gaussian} can be calculated as $M_{\text{RVD}}=q(\delta_{\text{max}})$.\\
In the one-dimensional case 
\begin{align*}
    M_{\text{RVD}}=\frac{\hat{\sigma}}{\sigma} \exp{\left( \frac{1}{2}\frac{\left(\hat{\mu}-\mu \right)^2}{\hat{\sigma}^2-\sigma^2}\right) }.
\end{align*}
\end{proof}
\subsection{Derivation of the expected violation probability of the scenario approach under the RVD}
\begin{proposition} \label{thm:Expected_RVD}
     Consider problem~\eqref{eq:RSP_for_SA} under assumption of convexity, uniqueness and existence of a solution, where the sample set $\sD_N$ has been drawn according to $\hat{\mathcal{P}}$ with the optimal solution $\hat{\vx}_{SA}(\sD_N)$. Assuming that the real uncertainty lies within an ambiguity set $\mathcal{A}$, bounded by the RVD with radius ${M_{\text{RVD}}}$, the expected probability of violation is
     \begin{multline} \label{eq:thm:expected}
    \E^N_{\hat{\mathcal{P}}}\{ V_\mathcal{P} (\hat{\vx}_s(\sD_N))\} \leq \sum_{i=d}^{N} f_B(N,i, \frac{1}{M_{\text{RVD}}}) \frac{d}{i+1}\\
     +\sum_{i=0}^{d-1} f_B(N,i,M_{\text{RVD}}^{-1}).
     \end{multline}
\end{proposition}
\begin{proof}
    To calculate the worst case mean of $V_{\mathcal{P}}(\hat{\vx}_s(\sD_N))$ for ambiguity measured by the RVD, we use Corollary~\ref{Corollary:Bound:Expected:ByParts}.
    $1-F_N^{\text{SA}} (\hat{\epsilon}_{M_{\text{RVD}}}(\epsilon))$ is continuously differentiable with the corresponding function derived with~\eqref{eq:SA_density} and the chain rule:
    \begin{equation*}
        f^{\text{SA}}_{\text{RVD}}(\epsilon)=\frac{d}{M_{\text{RVD}}}\binom{N}{d}\left( \frac{\epsilon}{M_{\text{RVD}}}\right)^{d-1}(1-\frac{\epsilon}{M_{\text{RVD}}})^{N-d}.
    \end{equation*}
    Note that $F_N^{\text{SA}} (\hat{\epsilon}_{M_{\text{RVD}}}(1))=F_N^{\text{SA}}(M_{\text{RVD}}^{-1})$, so reformulating the integral in Corollary~\ref{Corollary:Bound:Expected:ByParts} leads to
\begin{multline}
    \int_0^1 d \binom{N}{d} (\frac{\epsilon}{M_{\text{RVD}}}) ^d \left(1-\frac{\epsilon}{M_{\text{RVD}}} \right)^{N-d} d\epsilon\\
    = \int_0^1  d \binom{N}{d} (\frac{\epsilon}{M_{\text{RVD}}})^d  \left((1-\frac{1}{M_{\text{RVD}}}) +\frac{1}{M_{\text{RVD}}} (1-\eps) \right)^{N-d} d\epsilon \\
    = \int_0^1  \sum_{i=0}^{N-d} d \binom{N}{d} \binom{N-d}{i} \frac{\epsilon^{d}}{M_{\text{RVD}}^{i+d}} (1-\epsilon)^{i}(1-\frac{1}{M_{\text{RVD}}})^{N-d-i} d\epsilon\\
    = \int_0^1 \sum_{i=d}^{N} \binom{N}{i} \frac{1}{M_{\text{RVD}}^i}(1-\frac{1}{M_{\text{RVD}}})^{N-i} \epsilon d\binom{i}{d}\epsilon^{d-1}(1-\epsilon)^{i-d}   d\epsilon\\
    = \sum_{i=d}^{N} \binom{N}{i} \frac{1}{M_{\text{RVD}}^i} (1-\frac{1}{M_{\text{RVD}}})^{N-i} \int_0^1 \epsilon f^{\text{SA}}_{\hat{\mathcal{P}}}(\epsilon)   d\epsilon\\
    =\sum_{i=d}^{N} f_B(N,i, \frac{1}{M_{\text{RVD}}}) \frac{d}{i+1}.
\end{multline}
The second equality is a reformulation of the last bracket, for which the third equality applies the binomial formula. The fourth equality is derived via an index shift, while for the last equalities~the binomial distribution and~\eqref{eq:mean_eps_scen_app} are used respectively. Using the reformulated integral, we obtain~\eqref{eq:thm:expected}.
\end{proof}
Proposition~\ref{thm:Expected_RVD} gives an analytical solution to the distributionally robust bounds on the expected probability of violation for ambiguity measured by the RVD. Similar to the RVD being able to produce feasible two-level results even for low risk levels, the corresponding one-level bound can be interpreted as a weighted mean of~\eqref{eq:mean_eps_scen_app} for sample sizes $d\leq \hat{N}\leq N$ plus a constant for $\hat{N}<d$.
}

\end{document}